\documentclass{article}%
\usepackage{amssymb}
\usepackage{amsfonts}
\usepackage{amsmath}
\usepackage{graphicx}
\usepackage{amsfonts}%
\providecommand{\U}[1]{\protect\rule{.1in}{.1in}}

\def\U{{\cal U }}

\def\b0{{\overline{0}}}

\newtheorem{theorem}{Theorem}[section]

\newtheorem{corollary}[theorem]{Corollary}

\newtheorem{definition}[theorem]{Definition}

\newtheorem{lemma}[theorem]{Lemma}

\newtheorem{proposition}[theorem]{Proposition}
\newtheorem{remark}[theorem]{Remark}

\newenvironment{proof}[1][Proof]{\noindent\textbf{#1.} }{\ \rule{0.5em}{0.5em}}
\begin{document}

\title{Mean sensitive, mean equicontinuous and almost periodic functions for
dynamical systems}
\author{Felipe~Garc\'ia-Ramos
\and Brian Marcus}
\date{}
\maketitle

\begin{abstract}
We show that a continuous abelian action (in particular $\mathbb{R}^{d}$) on a
compact metric space equipped with an invariant ergodic measure has discrete
spectrum if and only it is $\mu-$mean equicontinuous (proven for
$\mathbb{Z}^{d}$ in \cite{weakeq}). In order to do this we introduce mean
equicontinuity and mean sensitivity with respect to a function. We study this
notion in the topological and measure theoretic setting. In the measure
theoretic case we characterize almost periodic functions with these concepts
and in the topological case we show that weakly almost periodic functions are
mean equicontinuous (the converse does not hold). We compare our results with
some results in the theory of Delone dynamical systems and quasicrystals.

\end{abstract}
\tableofcontents

\bigskip

A $\mathbb{Z-}$ topological dynamical system ($\mathbb{Z-}$TDS) is a pair
$(X,T)$ where $X$ is a compact metric space (with metric $d$) and
$T:X\rightarrow X$ a continuous invertible function (we will deal with abelian
$\mathbb{G}$ actions, but for the introduction using $\mathbb{Z}$ will be
enough to get intuition on the definitions). In the theory of topological
dynamical systems several notions of order have been studied. The most ordered
systems are the periodic systems, i.e. when for every $x\in X$ there exists
$n_{x}>0$ such that $T^{n_{x}}x=x;$ followed by the equicontinuous systems
i.e. when for every $\varepsilon>0$ there exists $\delta>0$ such that if
$d(x,y)\leq\delta$ then $d(T^{i}x,T^{i}y)\leq\varepsilon$ for every $i\geq0$.
This notion means that the system if highly predictable in the following
sense: if you only know $x$ is inside a small $\delta-$neighbourhood you will
be able to predict with $\varepsilon-$precision the orbit of $x.$ Several
weaker notions like distality, nullness, tameness, mean equicontinuity, among
others have been studied. In particular we are interested in mean
equicontinuity. A system is mean equicontinuous , if for every $\varepsilon>0$
there exists $\delta>0$ such that if $d(x,y)\leq\delta$ then $\lim\sup\frac
{1}{n}%
{\textstyle\sum\nolimits_{i=1}^{n}}
d(T^{i}x,T^{i}y)\leq\varepsilon$ . Here we have predictability not for every
$i$ but on sets of high density (see Proposition \ref{eq_metrics}). According
to Asulander \cite{auslander1959mean}, Fomin introduced these systems in
\cite{fominoriginal}.

Mean equicontinuity has received interest in recent years due to connections
with ergodic properties of measurable dynamical systems, i.e. dynamical
systems equipped with an invariant probability measure. In particular, it has
been shown that using a measure theoretic version of mean equicontinuity one
can characterize when a measure-preserving system has discrete spectrum
\cite{weakeq} (i.e. when the eigenfunctions of the Koopman operator generate
$L^{2}$) and when the maximal equicontinuous factor is actually an isomorphism
(i.e. when the continuous eigenfunctions of the Koopman operator generate
$L^{2}$) \cite{li2013mean, downarowicz2016isomorphic} . The proof of the
characterization of discrete spectrum using mean equicontinuity in
\cite{weakeq} does not hold for continuous actions (e.g. $\mathbb{R}^{d}$). In
this paper we prove that the characterization of discrete spectrum also holds
for abelian $\mathbb{G}$-actions (Theorem \ref{meds}). To do this we introduce
mean equicontinuity with respect to a function. We say that $(X,T)$ is mean
equicontinuous with respect to $f:X\rightarrow\mathbb{C}$ if for every
$\varepsilon>0$ there exists $\delta>0$ such that if $d(x,y)\leq\delta$ then
$\lim\sup\frac{1}{n}%
{\textstyle\sum\nolimits_{i=1}^{n}}
\left\vert f(T^{i}x)-f(T^{i}y)\right\vert \leq\varepsilon.$ Actually we will
see that a system is mean equicontinuous if and only if it is mean
equicontinuous with respect to every continuous function. Using the measure
theoretic version of this concept we can characterize almost periodic
functions (Corollary \ref{corhap}). Almost periodic functions are classical
objects in ergodic theory that are used to characterize discrete spectrum and
weak mixing.

\bigskip Another important concept that appears in the paper is sensitivity. A
topological dynamical system is \textit{sensitive} if there exists
$\varepsilon>0$ such that for every $x,y\in X$ there exists $i$ such that
$d(T^{i}x,T^{i}y)\geq\varepsilon$ and \textit{mean sensitive} if there exists
$\varepsilon>0$ such that for every $x,y\in X$ we have $\lim\sup\frac{1}{n}%
{\textstyle\sum\nolimits_{i=1}^{n}}
d(T^{i}x,T^{i}y)\geq\varepsilon.$

We study similar concepts in three different categories of dynamical systems
(measurable, topological, and measurable and topological) which divide
sections of the paper. The paper is arranged as follows.

In Section \ref{measure} we consider measure preserving systems (MPS) without
any topology. The main objective of this section is to characterize almost
periodic functions. To do this we introduce a purely measure theoretic notion
(no topology) of mean sensitivity with respect to an $L^{2}$ function $f$
($\mu-f-$mean sensitivity see Definition \ref{meansensitive}). We then
characterize this concept as the negation of almost periodicity of $f$ (when
the action is Abelian). As a corollary, we characterize systems with discrete
spectrum as those ergodic systems that are not mean sensitive with respect to
all $L^{2}$ functions; and we characterize weakly mixing systems as those
ergodic systems which are mean sensitive with respect to all non-constant
$L^{2}$ functions. Note that while it is possible to develop a purely measure
theoric version of mean sensitivity, we did not find a natural purely measure
theoretic notion of mean equicontinuity (other than negation of mean sensitivity).

In Section \ref{topological} we consider topological dynamical systems (TDS)
(without measures). We define the purely topological versions of mean
equicontinuity and mean sensitivity of a TDS relative to a given continuous
function $f$. We show that in the case of a minimal TDS, mean sensitivity and
mean equicontinuity (relative to $f$) are complementary notions and partition
$C(X)$. Finally, for a TDS we give results on the relationships between
topological almost periodicity defined by Ellis in
\cite{ellis1959equicontinuity}, and mean equicontinuity. In particular we show
that if a function is weakly almost periodic then the system is mean
equicontinuous with respect to the function (the converse does not hold). \ 

In Section \ref{hybrid} we consider topological dynamical systems equipped
with invariant probability measures. It is in this category where we can
define a hybrid (measure theoretic and topological) form of mean
equicontinuity relative to a given function $f$ and a measure $\mu$ ($\mu
-f-$mean equicontinuity). First we show that this notion is complementary to
the measure-theoretic notion of mean sensitivity relative to $f$ (defined in
Section \ref{measure}). Finally we present a characterization of discrete
spectrum for TDS equipped with ergodic measures (Theorem \ref{meds}). In
Section 4 we apply this result to characterize quasicrystalline behaviour on
Delone sets. Finally we compare our results with other characterizations of
discrete spectrum in terms of topological averages (e.g. \cite{quasicrystals}%
\cite{lenz2016autocorrelation}).

\bigskip In this paper $\mathbb{G}$ denotes an abelian locally compact group
(with operation $+$ and identity $e$). Every locally compact group has a
unique measure that is invariant under rotations, known as Haar measure on
$\mathbb{G}$. We will denote this measure with $\nu.$ In particular if
$\mathbb{G=R}^{d},$ then $\nu$ is the Lebesgue measure, and if\ $\mathbb{G}$
countable $\nu$ is the cardinality$.$

\bigskip A sequence of measurable sets $\left\{  F_{n}\right\}  _{n\in
\mathbb{N}}$ $\subset\mathbb{G}$ with finite non-zero $\nu-$measure is a
\textbf{F\o lner sequence} if for every $g\in\mathbb{G}$
\[
\lim_{n\rightarrow\infty}\frac{\nu(gF_{n}\bigtriangleup F_{n})}{\nu(F_{n}%
)}=0.
\]

\bigskip We say a group $\mathbb{G}$ is \textbf{amenable} if it contains a
F\o lner sequence. Every Abelian group is amenable. In particular if
$\mathbb{G}=\mathbb{Z}^{d}$ or $\mathbb{G}=\mathbb{R}^{d}$ then $F_{n}:=$
$\left[  -n,n\right]  ^{d}$ is a F\o lner sequence.

\begin{definition}
Let $S\subset\mathbb{G}$ and $\left\{  F_{n}\right\}  _{n\in\mathbb{N}}$ a
F\o lner sequence. We define the \textbf{lower density of S} by%
\[
\underline{D}(S):=\liminf_{n\rightarrow\infty}\frac{\nu(S\cap F_{n})}%
{\nu(F_{n})},
\]
and the \textbf{upper density of S } by
\[
\overline{D}(S):=\limsup_{n\rightarrow\infty}\frac{\nu(S\cap F_{n})}{\nu
(F_{n})}.
\]

\end{definition}


\section{Measure theoretical results}

\label{measure}

\bigskip
A $\mathbb{G}-$ \textbf{measure preserving system (}$\mathbb{G}-$\textbf{MPS)}
is a quadruple $(X,\Sigma,\mu,T)$ where $(X,\Sigma,\mu)$ is a standard
probability space and
\[
T:\mathbb{G}\times X\rightarrow X,(j,x)\rightarrow T^{j}x
\]
is a group action (that is $T^{e}x=x$ and $T^{g}T^{j}x=T^{g+j}x$ $\forall x\in
X$)$\ $and $T^{j}:X\rightarrow X$ is measure preserving on $(X,\Sigma,\mu)$
for every $j\in\mathbb{G}$. When it is not needed we will omit writing
$\Sigma.$ A $\mathbb{Z}$-measure preserving system is generated by a measure
preserving invertible transformation $T:(X,\mu)\rightarrow(X,\mu)$ on a
Lebesgue probability space $(X,\mu)$.

We say a measurable subset is invariant if it is invariant under every
$T^{i};$ we say a \textbf{$\mathbb{G-}$}MPS is an \textbf{ergodic} if every
measurable invariant set has measure 0 or 1. The collection of measurable sets
with positive $\mu$-measure is denoted $\Sigma^{+}.$

A measure preserving system $(X,\mu,T)$ generates a family of unitary linear
operators (known as the Koopman operators) on the complex Hilbert space
$L^{2}(X,\mu),$ $U_{T}^{j}:f\mapsto f\circ T^{j}.$ We denote the inner product
of $L^{2}(X,\mu)$ by $\left\langle \cdot,\cdot\right\rangle $.

\begin{definition}
\label{generic}Let $(X,\mu,T)$ be an ergodic $\mathbb{G}-$MPS and $\left\{
F_{n}\right\}  $ a F\o lner sequence.

Let $A\subset X$ be a measurable set. We say $x\in X$ is a \textbf{generic
point for }$A$\textbf{\ }if
\[
\lim_{n\rightarrow\infty}\frac{\nu(\left\{  i\in F_{n}:T^{i}x\in A\right\}
)}{\nu(F_{n})}=\mu(A).
\]

Let $f$ be a measurable function. We say $x\in X$ is a \textbf{generic point
for }$f$\textbf{\ }if
\[
\lim_{n\rightarrow\infty}\frac{1}{\nu(F_{n})}\int_{i\in F_{n}}f(T^{i}%
x)d\nu=\int_{X}fd\mu(x).
\]

\end{definition}

\bigskip Lindenstrauss proved that for every amenable group there exists a
F\o lner sequence $\left\{  F_{n}\right\}  $ that satisfies the pointwise
ergodic theorem i.e. that for every $f\in L^{1}(X,\mu)$ almost every $x\in X$
is a generic point for $f$ \cite{lindenstrauss2001pointwise}$.$

\begin{remark}
From now on we associate to every $\mathbb{G}$ a F\o lner sequence $\left\{
F_{n}\right\}  $ that satisfies the pointwise ergodic theorem.
\end{remark}


\subsection{Almost periodic and $\mu-$mean sensitive functions}

Let $S^{1}\subset\mathbb{C}$ be the unit circle. Given $\mathbb{G}$ we denote
with $\widehat{\mathbb{G}}$ the (Pontryagin) dual group, which is composed by
the continuous group homomorphisms from $\mathbb{G}$ to $S^{1}$. Given
$g\in\mathbb{G}$ and $w\in\widehat{\mathbb{G}}$, we define $\left(
w,j\right)  :=w(j).$

\begin{definition}
\label{eigen}Let $(X,\mu,T)$ be an ergodic $\mathbb{G}-$MPS and $f\in
L^{2}(X,\mu)$ (complex-valued functions).

$\cdot$We say $f$ is an \textbf{almost periodic function} if $cl\left\{
U^{j}f:j\in\mathbb{G}\right\}  $ is compact as a subset of $L^{2}(X,\mu)$. We
denote the set of almost periodic functions by $H_{ap}$.

$\cdot$We say $f$ $\neq0$ is an \textbf{eigenfunction} of $(X,\mu,T)$ if there
exists $w\mathcal{\in}\widehat{\mathbb{G}}$ such that
\[
U^{j}(f)=(w,j)f\text{ \ }\forall j\in\mathbb{G}.
\]

\end{definition}

\begin{remark}
When $d=1,$ the action reduces to the action of a single measure preserving
transformation $T$ on $(X,\mu)$ and the definition of eigenfunction reduces to
the usual definition of an eigenfunction for a unitary operator: $f$ $\neq0$
and there exists $\lambda\mathcal{\in}\mathbb{C}$ such that $\left\vert
\lambda\right\vert =1$ and
\[
f\circ T=\lambda f.
\]

\end{remark}

\begin{definition}
A subset $S\subset\mathbb{G}$ is \textbf{syndetic }if there exists a bounded
set $K\subset\mathbb{G}$ such that $S+K=\mathbb{G}$.
\end{definition}

The following results are well known (see for example
\cite{furstenberg2014recurrence}\cite{furstenberg1982ergodic}).

\begin{proposition}
\label{prop}\bigskip Let $(X,\mu,T)$ be a $\mathbb{G-}$MPS.

\begin{enumerate}
\item $H_{ap}$ is the closure of the set spanned by the eigenfunctions of $T.$

\item $f\in L^{2}(X,\mu)$ is almost periodic if and only if for every
$\varepsilon>0$ there exists a syndetic set $S\subset\mathbb{G}$ such that
$\int\left\vert f-U^{j}f\right\vert ^{2}d\mu\leq\varepsilon$ for every $j\in
S$.

\item The product of two almost periodic functions is almost periodic.

\item If $(X,\mu,T)$ is ergodic and $f$ is an eigenfunction then $\left\vert
f\right\vert $ is constant almost everywhere.
\end{enumerate}
\end{proposition}

\begin{definition}
\bigskip Let $(X,\mu,T)$ and $(X^{\prime},\mu^{\prime},T^{\prime})$ be two
$\mathbb{G-}$MPS. We say they are\textbf{\ isomorphic} if there exists an a.e.
bijective and measure preserving function $f:(X,\mu)\rightarrow(X^{\prime}%
,\mu^{\prime})$ such that $T^{\prime}\circ f=f\circ T$ and $f^{-1}$ is measure preserving.
\end{definition}

\begin{definition}
\label{ds} Let $(X,\mu,T)$ be an ergodic $\mathbb{G-}$MPS.

$\cdot(X,\mu,T)$ has \textbf{discrete spectrum} if there exists an orthonormal
basis for $L^{2}(X,\mu)$ which consists of eigenfunctions of $(X,\mu,T).$

$\cdot(X,\mu,T)$ has \textbf{continuous spectrum} if the only eigenfunctions
are the constant functions.

$\cdot(X,\mu,T)$ is a \textbf{measurable isometry }if it is isomorphic to
$(X^{\prime},\mu^{\prime},T^{\prime})$ where $(X^{\prime},T^{\prime})$ is
where $X^{\prime}$ is a compact metric space and each element of $T^{\prime}$
is an isometry.

$\cdot(X,\mu,T)$ is \textbf{weakly mixing} if and only if of $(X\times
X,\mu\times\mu,T\times T)$ is an ergodic $\mathbb{G-}$MPS.
\end{definition}

The following two results are due to Halmos and Von-Neumann (see for example
\cite{walters2000introduction}:\ note that the original proof is stated for
$\mathbb{G=Z},$ nonetheless the same proof is valid for abelian group actions).

\begin{theorem}
\label{discrete}Let $(X,\mu,T)$ be an ergodic system. The following are equivalent:

$\cdot(X,\mu,T)$ has discrete spectrum.

$\cdot(X,\mu,T)$ is a measurable isometry.

$\cdot L^{2}(X,\mu)=H_{ap}$
\end{theorem}

\begin{theorem}
\label{weak}Let $(X,\mu,T)$ be an ergodic system. The following are equivalent:

$\cdot(X,\mu,T)$ has continuous spectrum.

$\cdot(X,\mu,T)$ is weakly mixing.

$\cdot H_{ap}$ consists only of constant functions.
\end{theorem}

Let $\mathcal{F}$ denote the set of all functions $\alpha:{\mathbb{G}%
}\rightarrow{\mathbb{C}}$ such that
\[
\lim\sup\left(  \frac{1}{\nu(F_{n})}\int_{F_{n}}\left\vert \alpha
(j)\right\vert ^{2}d\nu(j)\right)  <\infty.
\]
We define the following pseudometric on $\mathcal{F}$:
\[
e(\alpha,\beta):=\lim\sup\left(  \frac{1}{\nu(F_{n})}\int_{F_{n}}\left\vert
\alpha(j)-\beta(j)\right\vert ^{2}d\nu(j)\right)  ^{1/2}%
\]
It is not difficult to show that this is indeed a pseudometric (i.e. satisfies
symmetry and the triangle inequality).

For any $f\in L^{2}(X,\mu)$ we introduce the following pseudometric on a set
of full measure in $X$.

\begin{definition}
\label{L2metric} Let $f\in L^{2}(X,\mu).$ We define
\[
d_{f}(x,y):=\lim\sup\left(  \frac{1}{\nu(F_{n})}\int_{F_{n}}\left\vert
f(T^{j}x)-f(T^{j}y)\right\vert ^{2}d\nu(j)\right)  ^{1/2}.
\]

\end{definition}

Using Cauchy-Schwartz and the pointwise ergodic theorem on can check that
$d_{f}(x,y)$ is finite-valued for all $x,y\in X$ that are generic points for
$f^{2}$. Letting $\alpha(j)=f(T^{j}x),\beta(j)=f(T^{j}y)$, we see that
$d_{f}(x,y)=e(\alpha,\beta)$ and thus it is a pseudo-metric on the set of
generic points for $f^{2}$. Whenever we use the pseudometric $d_{f}$ we only
consider the points that are generic for $f^{2}.$

\begin{definition}
\label{meansensitive} Let $f\in L^{2}(X,\mu).$ We say an $\mathbb{G-}$MPS
$(X,\mu,T)$ is $\mu-f-$\textbf{mean sensitive} if there exists $\varepsilon>0$
such that for every $A\in\Sigma^{+}$ there exist $x,y\in A$ such that
\begin{equation}
d_{f}(x,y)>\varepsilon. \label{dfe}%
\end{equation}
In this case we also say that $f$ is $\mu-$\textbf{mean sensitive} and
$\varepsilon$ is a $\mu$-\textbf{sensitivity constant} for $f$. We denote the
set of $\mu-$mean sensitive functions by $H_{ms}$.
\end{definition}

\begin{remark}
We will see in Section~\ref{equivalent_metrics} that for $f \in L^{\infty
}(X,\mu)$, there are equivalent pseudometrics that give different perspectives
on mean sensitivity.
\end{remark}




\begin{lemma}
\label{open}Let $(X,\mu,T)$ be an ergodic $\mathbb{G-}$MPS. Then
$H_{ms}\subset L^{2}(X,\mu)$ is an open set.
\end{lemma}

\begin{proof}
Let $f\in H_{ms}$ and let $\varepsilon>0$ be a mean sensitivity constant for
$f$.


Let $g\in L^{2}(X,\mu)$ such that $\int\left\vert f-g\right\vert ^{2}d\mu
\leq(\varepsilon/4).$
By the pointwise ergodic theorem there exists $Y\subset X$ with $\mu(Y)=1$
and
\[
\lim\frac{1}{\nu(F_n)}\int_{F_{n}}\left\vert
f(T^{j}x)-g(T^{j}x)\right\vert ^{2}d\nu(j)\leq(\varepsilon/8)\text{ for every
}x\in Y.
\]
For every $A\in\Sigma^{+}$ there exist $x,y\in A\cap Y$ such that%
\[
\lim\sup\frac{1}{\nu(F_n)}\int_{F_{n}}\left\vert
f(T^{j}x)-f(T^{j}y)\right\vert ^{2}d\nu(j)>\varepsilon.
\]
Since $e(\alpha,\beta)$ is a pseudometric we obtain
\begin{align*}
&  \lim\sup\left(  \frac{1}{\nu(F_n) }\int_{F_{n}%
}\left\vert g(T^{j}x)-g(T^{j}y)\right\vert ^{2}d\nu(j)\right)  ^{1/2}\\
&  \geq\lim\sup\left(  \frac{1}{\nu(F_n) }\int_{F_{n}%
}\left\vert f(T^{j}x)-f(T^{j}y)\right\vert ^{2}d\nu(j)\right)  ^{1/2}\\
&  -\lim\left(  \frac{1}{\nu(F_n) }\int_{F_{n}}\left\vert
f(T^{j}x)-g(T^{j}x)\right\vert ^{2}d\nu(j)\right)  ^{1/2}\\
&  -\lim\left(  \frac{1}{\nu(F_n) }\int_{F_{n}}\left\vert
f(T^{j}y)-g(T^{j}y)\right\vert ^{2}d\nu(j)\right)  ^{1/2}\\
&  >\sqrt{\varepsilon}(1-1/\sqrt{2}).
\end{align*}
Thus $g\in H_{ms}$ and so $H_{ms}\subset L^{2}(X,\mu)$ is open.
\end{proof}


\begin{theorem}
\label{main copy(1)}Let $(X,\mu,T)$ be an ergodic $\mathbb{G-}$MPS. Then
$H_{ap}^{c}=H_{ms}.$

\end{theorem}

\begin{proof}
We first show that $H_{ap}^{c}\subseteq H_{ms}.$

Assume that $f \in H_{ap}^{c}$, and so $cl\left\{  U^{j}f:j\in\mathbb{G}%
\right\}  $ is not compact. This implies it is not totally bounded, hence
there exists $\varepsilon>0$ and an infinite subset $S\subset\mathbb{G}$ such
that
\begin{equation}
\int\left\vert U^{i}f-U^{j}f\right\vert ^{2}d\mu\geq\varepsilon\text{ for
every }i\neq j\in S
\end{equation}

Let $A\in\Sigma^{+}$. There exist $s\neq t\in S$ such that
\[
\mu(T^{-s}A\cap T^{-t}A)>0.
\]

Let $g(x):=\left\vert U^{s}f(x)-U^{t}f(x)\right\vert ^{2}.$ By the pointwise
ergodic theorem, there exists $z\in T^{-s}A\cap T^{-t}A$ such that
\[
\lim\frac{1}{\nu(F_n) }\int_{F_{n}}g(T^{j}z)d\nu(j)=\int
gd\mu\geq\varepsilon.
\]

Let $p:=T^{s}z\in A$ and $q:=T^{t}z\in A.$ Since $\mathbb{G}$ is commutative
we have that
\begin{align*}
&  \lim\frac{1}{\nu(F_n) }\int_{F_{n}}\left\vert
f(T^{j}p)-f(T^{j}q)\right\vert ^{2}d\nu(j)\\
&  =\lim\frac{1}{\nu(F_n) }\int_{F_{n}}\left\vert
f(T^{j+s}z)-f(T^{j+t}z)\right\vert ^{2}d\nu(j)\\
&  =\lim\frac{1}{\nu(F_n) }\int_{F_{n}}g(T^{j}%
z)d\nu(j)>\varepsilon.
\end{align*}

We conclude that for every $A\in\Sigma^{+}$ there exist $p,q\in A$ such that
\[
\lim\frac{1}{\nu(F_n) }\int_{F_{n}}\left\vert
f(T^{j}p)-f(T^{j}q)\right\vert ^{2}d\nu(j)>\varepsilon.
\]


Now we show that $H_{ap} \subseteq H_{ms}^{c}.$

Let $\varepsilon>0$ and $g$ be a measurable function. There exists a set
$X_{\varepsilon}\in\Sigma_{+}$ such that $\left\vert g(x)-g(y)\right\vert
^{2}\leq\varepsilon$ \ for every $x,y\in X_{\varepsilon}.$

Now assume that $g$ is an eigenfunction and so there exists $w\mathcal{\in
}\mathbb{R}^{d}$ such that $U^{j}(g)=e^{2\pi i\left\langle w,j\right\rangle
}g.$ Then for every $x,y\in X_{\varepsilon}$ and every $j\in\mathbb{G}$ we
have that
\begin{align*}
&  \left\vert g(T^{j}x)-g(T^{j}y)\right\vert ^{2}\\
&  =\left\vert e^{2\pi i\left\langle w,j\right\rangle }(g(x)-g(y))\right\vert
^{2}\\
&  =\left\vert g(x)-g(y)\right\vert ^{2}\leq\varepsilon.
\end{align*}
Thus $g\in H_{ms}^{c}.$

Using Lemma \ref{open} and Proposition~\ref{prop} (part 1), we get that
$H_{ap}\subset H_{ms}^{c},$ as desired.

\end{proof}

Note that the proof above shows that in the definition of $H_{ms}$ we can
replace $\limsup$ with $\lim$, i.e., $f\in H_{ms}$ if and only if there exists
$\varepsilon>0$ such that for every $A\in\Sigma^{+}$ there exists $x,y\in A$
such that%
\[
\lim\frac{1}{\nu(F_n) }\int_{F_{n}}\left\vert
f(T^{j}x)-f(T^{j}y)\right\vert ^{2}d\nu(j)>\varepsilon.
\]


\begin{corollary}
\label{propo}\bigskip Let $(X,\mu,T)$ be an ergodic ergodic $\mathbb{G-}$MPS.

$1)$ $(X,\mu,T)$ has discrete spectrum if and only if $L^{2}(X,\mu)=H_{ms}%
^{c}$

$2)$ $(X,\mu,T)$ is weakly mixing if and only if $f\in H_{ms}$ for all
non-constant $f\in L^{2}(X,\mu)$ .

\end{corollary}


\begin{proof}
$1)$ Follows from Theorems \ref{discrete} and~\ref{main copy(1)}.

$2)$ Follows from Theorems \ref{weak} and~\ref{main copy(1)}.
\end{proof}


\subsection{$\mu-f-$Mean expansivity}

The proof of Theorem \ref{strongsensitive} is analogous to the proof of
Theorem 2.6 of \cite{weakeq} but uses $d_{f}$ instead of the Besicovitch
pseudometric. The proof is short so we write it for completeness.




\begin{definition}
\label{expansive}Let $f\in L^{2}(X,\mu).$ A $\mathbb{G-}$MPS is $\mu
-f-$\textbf{mean expansive} if there exists $\varepsilon>0$ such that
$\mu\times\mu(\left\{  (x,y):d_{f}(x,y)>\varepsilon\right\}  )=1.$ We say that
$\varepsilon$ is a $\mu$-\textbf{mean expansive constant} for $f$.
\end{definition}

\begin{lemma}
\label{ecu}Let $(X,\mu,T)$ be an ergodic $\mathbb{G-}$MPS. Then
\[
g(x):=\mu(\left\{  y:d_{f}(x,y)\leq\varepsilon\right\}  )
\]
$\ $is constant for almost every $x\in X$ and equal to $\mu\times\mu\left\{
(x,y):d_{f}(x,y)\leq\varepsilon\right\}  .$
\end{lemma}

\begin{proof}
It is not hard to see that $d_{f}(x,y)$ is $\mu\times\mu-$measurable. This
means that $\left\{  (x,y):d_{f}(x,y)\leq\varepsilon\right\}  $ is $\mu
\times\mu-$measurable for every $\varepsilon>0$. Using Fubini's Theorem we
obtain that
\begin{align*}
\mu\times\mu\left\{  (x,y):d_{f}(x,y)\leq\varepsilon\right\}   &  =\int
_{X}\int_{X}1_{\left\{  (x,y):d_{f}(x,y)\leq\varepsilon\right\}  }d\mu
(y)d\mu(x)\\
&  =\int_{X}\mu\left\{  y:d_{f}(x,y)\leq\varepsilon\right\}  d\mu(x).
\end{align*}

Since $g$ is $T$-invariant we conclude that $g(x)$ is constant for almost
every $x\in X$ and equal to $\mu\times\mu\left\{  (x,y):d_{f}(x,y)\leq
\varepsilon\right\}  .$
\end{proof}

\begin{definition}
\label{meansensitive2}Let $B_{\varepsilon}^{f}(x):=\left\{  y:d_{f}%
(x,y)\leq\varepsilon\right\}  .$
\end{definition}

\begin{theorem}
\label{strongsensitive}Let $(X,\mu,T)$ be an ergodic system and $f\in
L^{2}(X,\mu).$ The following are equivalent:

$1)$ $(X,T)$ is $\mu-f-$mean sensitive.

$2)$ $(X,T)$ is $\mu-f-$mean expansive$.$

$3)$ There exists $\varepsilon>0$ such that for almost every $x,$
$\mu(B_{\varepsilon}^{f}(x))=0.$
\end{theorem}

\begin{proof}
$2)\Rightarrow1)$ Let $\varepsilon$ be a $\mu$-mean expansive constant for
$f$. Let $A\in\Sigma^{+}$ and so $A\times A\in(\Sigma\times\Sigma)^{+}.$ By
hypothesis we can find $(x,y)\in A\times A$ such that $d_{f}(x,y)>\varepsilon
.$

$1)\Rightarrow3)$

Suppose $(X,T)$ is $\mu-f-$mean sensitive (with $\mu-f-$mean sensitivity
constant $\varepsilon$) and that $3)$ is not satisfied. This means there
exists $x\in X$ such that $\mu(B_{\varepsilon/2}^{f}(x)) > 0.$
For any $y,z \in B_{\varepsilon/2}^{f}(x)$,
we have that $d_{f}(y,z)\le\varepsilon.$ This contradicts the assumption that
$(X,T)$ is $\mu-f-$mean sensitive.

$3)\Rightarrow2)$

Using Lemma \ref{ecu} we obtain that $\mu\times\mu\left\{  (x,y):d_{f}%
(x,y)\leq\varepsilon\right\}  =0.$
\end{proof}

\medskip


\subsection{Equivalent pseudometrics}

\label{equivalent_metrics}

We consider the following pseudometrics, which turn out to be equivalent to
$d_{f}$ when $f$ is bounded.

\begin{definition}
\label{L1metric} Let $f\in L^{2}(X,\mu).$ Define
\[
d_{f}^{\prime}(x,y):=\lim\sup\frac{1}{\nu(F_{n})}\int_{F_{n}}\left\vert
f(T^{j}x)-f(T^{j}y)\right\vert d\nu(j)
\]

\end{definition}

\begin{definition}
\label{density_metric} Let $f\in L^{2}(X,\mu).$ Define
\[
\rho_{f}(x,y):=\inf\left\{  \varepsilon>0:\overline{D}(\Delta_{\varepsilon
}(x,y))<\varepsilon\right\}  ,
\]
where
\[
\Delta_{\varepsilon}(x,y):=\left\{  i\in\mathbb{G}:\text{ }|f(T^{i}%
x)-f(T^{i}y)|>\varepsilon\right\}  .
\]

\end{definition}

\begin{proposition}
\label{eq_metrics} For an $\mathbb{G-}$MPS $(X,\mu,T)$ and $f\in L^{\infty
}(X,\mu)$, the pseudometrics $d_{f}(x,y),d_{f}^{\prime}(x,y)$, and $\rho
_{f}(x,y)$ are equivalent, i.e., generate the same topology.
\end{proposition}

\begin{proof}
Without loss of generality we may assume that $|f| \le1/2$. So, $(d_{f}%
(x,y))^{2} \le d_{f}^{\prime}(x,y)$.

Let $\varepsilon>0.$ It suffices to show that 1) if $\rho_{f}(x,y)<\varepsilon
/2$, then $d_{f}^{\prime}(x,y)<\varepsilon$ and 2) if $d_{f}^{2}%
(x,y)<\varepsilon^{3}$, then $\rho_{f}(x,y)<\varepsilon$.

Proof of 1): If $\rho_{f}(x,y)<\varepsilon/2$, then
\[
\overline{D}(\left\{  j\in\mathbb{G}:|f(T^{j}x)-f(T^{j}y)|>\varepsilon
/2\right\}  )<\varepsilon/2.
\]
So,
\begin{align*}
&  d_{f}^{\prime}(x,y)\\
&  =\lim\sup\frac{1}{\left\vert F_{n}\right\vert }\int_{F_{n}}|f(T^{j}%
x)-f(T^{j}y)|d\nu(j)\\
&  \leq\frac{\varepsilon}{2}\overline{D}(\left\{  j\in\mathbb{G}%
:|f(T^{j}x)-f(T^{j}y)|>\varepsilon/2\right\}  )+\overline{D}(\left\{
j\in\mathbb{G}:|f(T^{j}x)-f(T^{j}y)|\leq\varepsilon/2\right\}  )\\
&  <\varepsilon.
\end{align*}

Proof of 2): Assume that $d_{f}^{2}(x,y)<\varepsilon^{3}$.
Suppose that $\rho_{f}(x,y)\geq\varepsilon$.
Then for all $\delta<\varepsilon$, we obtain
\begin{align*}
&  \lim\sup\frac{1}{\left\vert F_{n}\right\vert }\int_{F_{n}}|f(T^{j}%
x)-f(T^{j}y)|^{2}d\nu(j)\\
&  \geq\delta^{2}\overline{D}(\left\{  j\in\mathbb{G}:|f(T^{j}x)-f(T^{j}%
y)|^{2}>\delta^{2}\right\}  )\\
&  =\delta^{2}\overline{D}(\left\{  j\in\mathbb{G}:|f(T^{j}x)-f(T^{j}%
y)|>\delta\right\}  )\\
&  \geq\delta^{3}.
\end{align*}
Thus
\[
\lim\sup\frac{1}{\nu(F_n) }\int_{F_{n}}|f(T^{j}%
x)-f(T^{j}y)|^{2}d\nu(j)\geq\varepsilon^{3}%
\]
yields a contradiction. We conclude that $\rho_{f}(x,y)<\varepsilon$.

\end{proof}

This implies we can use $d_{f}^{\prime}(x,y)$ or $\rho_{f}(x,y)$ as
alternatives to $d_{f}(x,y)$ in the definition of $\mu$-$f$-mean sensitivity.

\section{Topological results}

\label{topological}

A $\mathbb{G-}$\textbf{topological dynamical system (TDS)} is a pair $(X,T),$
where $X$ is a compact metric space and $T$ is a group action with
$T^{j}:X\rightarrow X$ continuous for every $j\in\mathbb{G}.$

The closed $\varepsilon$-balls of $X$ will be denoted by $B_{\varepsilon}(x)$,
and the collection of Borel sets of $X$ by $\mathcal{B}_{X}$.

\subsection{Mean equicontinuity and mean sensitivity}


\begin{definition}
\label{besi}We define
\[
d_{b}(x,y):=\limsup_{n\rightarrow\infty}\frac{1}{\left\vert F_{n}\right\vert
}\int_{F_{n}}d(T^{j}x,T^{j}y)d\nu(j),
\]%
\[
\rho_{b}(x,y):=\inf\left\{  \varepsilon>0:\overline{D}(\Delta_{\varepsilon
}(x,y))<\varepsilon\right\}  ,
\]
where
\[
\Delta_{\varepsilon}(x,y):=\left\{  i\in\mathbb{G}:~d(T^{i}x,T^{i}%
y)>\varepsilon\right\}  .
\]

\end{definition}



Using subadditivity of $\limsup$, it is not hard to show that $d_{b}$ and
$\rho_{b}$ are indeed pseudometrics. The subscript \textquotedblleft%
$b$\textquotedblright\ stands for \textquotedblleft
Besicovitch,\textquotedblright\ as these are versions of the Besicovitch pseudometric.

\begin{lemma}
\label{eqps}Let $(X,T)$ be a $\mathbb{G-}$TDS. Then $d_{b}$ and $\rho_{b}$ are
equivalent pseudometrics.
\end{lemma}

\begin{proof}
The proof is nearly identical to the proof of equivalence of $d_{f}^{\prime}$
and $\rho_{f}$ in Proposition~\ref{eq_metrics}: simply replace $|f(T^{i}x) -
f(T^{i}y)|$ by $d(T^{i}x, T^{i}y)$ in the proof.
\end{proof}

\begin{definition}
Let $(X,T)$ be a $\mathbb{G-}$TDS. We say $(X,T)$ is \textbf{mean sensitive}
if for every non-empty open set $U$ there exist $x,y\in U$ such that
$d_{b}(x,y)>\varepsilon.$
\end{definition}

\begin{definition}
\label{mequi}Let $(X,T)$ be a $\mathbb{G-}$TDS. We say $x\in X$ is a
\textbf{mean equicontinuity point} if for every $\varepsilon>0$ there exists
$\delta>0$ such that if $y\in B_{\delta}(x)$ then
\[
d_{b}(x,y)<\varepsilon.
\]
We say $(X,T)$ is \textbf{mean equicontinuous (or mean-L-stable)} if every
$x\in X$ is a mean equicontinuity point. We say $(X,T)$ is \textbf{almost mean
equicontinuous }if the set of mean equicontinuity points is residual.
\end{definition}

Using Lemma~\ref{eqps}, we conclude that $(X,T)$ is mean sensitive if and only
if for every non-empty open set $U$ there exist $x,y\in U$ such that $\rho
_{b}(x,y)>\varepsilon;\ $and $x\in X$ is a mean equicontinuity point if and
only if for every $\varepsilon>0$ there exists $\delta>0$ such that if $y\in
B_{\delta}(x)$ then
\[
\rho_{b}(x,y)<\varepsilon.
\]

\begin{remark}
\label{uniform}
If $(X,T)$ is mean equicontinuous, then it is uniformly mean equicontinuous in
the sense that for every $\varepsilon>0$ there exists $\delta>0$ such that if
$d(x,y)\leq\delta$ then $\overline{D}(i\in\mathbb{G}:d(T^{i}x,T^{i}%
y)>\varepsilon)<\varepsilon.$ One way to see this is by identifying points
$x,y$ such that $d_{b}(x,y)=0$. Namely, we obtain a metric space $\left(
X/d_{b},d_{b}\right)  $ and a natural projection $\pi_{b}:(X,d)\rightarrow
(X/d_{b},d_{b})$ by mapping $x\in X$ to its equivalence class in $X/d_{b}$.
Then mean equicontinuity of $(X,T)$ means that $\pi_{b}$ is continuous and
hence uniformly continuous.
\end{remark}

Mean equicontinuous systems were introduced as mean-L-stable systems by Fomin
\cite{fominoriginal}. These systems have been studied in
\cite{auslander1959mean}, \cite{oxtoby1952}, \cite{scarpellini1982stability}
and \cite{li2013mean}.

In \cite{AkinAuslander}\cite{auslanderyorke} dichotomies between
equicontinuity and sensitivity were exhibited. In \cite{weakeq} (and
independently in \cite{li2013mean}) the following dichotomies involving mean
sensitivity and mean equicontinuity were proved.

\begin{definition}
Let $(X,T)$ be a $\mathbb{G-}$TDS. We say $x\in X$ is a \textbf{transitive
point} if $\left\{  T^{i}x:i\in\mathbb{G}\right\}  $ is dense. We say $(X,T)$
is \textbf{transitive} if $X$ contains a residual set of transitive points. If
every $x\in X$ is transitive then we say the system is \textbf{minimal}.
\end{definition}

\begin{theorem}
\label{sensi copy(1)}\cite{weakeq}\cite{li2013mean}A transitive system is
either almost mean equicontinuous or mean sensitive. A minimal system is
either mean equicontinuous or mean sensitive.
\end{theorem}

In the next subsection we will introduce versions of mean equicontinuity and
mean sensitivity relative to given $f\in C(X)$ and we will establish similar dichotomies.

\subsection{$f-$Mean equicontinuity and $f-$Mean sensitivity}


\begin{definition}
Let $(X,T)$ be a $\mathbb{G-}$TDS and $f\in C(X).$ We say $(X,T)$ is
$f-$\textbf{mean sensitive} if there exists $\varepsilon>0$ such that for
every non-empty open set $U$ there exists $x,y\in U$ such that
\[
d_{f}(x,y)>\varepsilon.
\]
In this case we also say that $f$ is \textbf{mean sensitive} for $(X,T)$, and
we denote the set of all mean sensitive functions for $(X,T)$ by $C_{ms}$.
\end{definition}

The notion of $f-$mean equicontinuity appeared in
\cite{scarpellini1982stability}.

\begin{definition}
Let $(X,T)$ be a $\mathbb{G-}$TDS and $f\in C(X).$ We say $x\in X$ is a
$f-$\textbf{mean equicontinuity point} if for every $\varepsilon>0$ there
exists $\delta>0$ such that if $d(x,y)\leq\delta$ then $d_{f}(x,y)\leq
\varepsilon.$
We say $(X,T)$ is $f-$\textbf{mean equicontinuous} if all $x\in X$ are
$f-$mean equicontinuous. In this case we also say that $f$ is \textbf{mean
equicontinuous function} for $(X,T)$; we denote the set of mean equicontinuous
functions by $C_{me}$. We say $(X,T)$ is $f-$ \textbf{almost mean
equicontinuous }if the set of mean equicontinuity points is residual.
\end{definition}

\begin{remark}
Since any continuous function on $X$ is bounded, we can use
Proposition~\ref{eq_metrics}, to characterize $f-$mean equicontinuity and
$f-$mean sensitivity by replacing $d_{f}(x,y)$ with $\rho_{f}(x,y)$ in the
definition. We will mainly use the latter.
\end{remark}

Again using the idea that a continuous function on a compact metric set is
uniformly continuous, it is not hard to see that if $(X,T)$ is $f-$mean
equicontinuous, then it is uniformly $f-$mean equicontinuous in the sense that
for every $\varepsilon>0$ there exists $\delta>0$ such that if $d(x,y)\leq
\delta$ then $d_{f}(x,y)\leq\varepsilon.$

\begin{definition}
Let $f\in C(X).$ We denote the set of $f-$mean equicontinuity points by
$E^{f}$ and we define
\[
E_{\varepsilon}^{f}:=\left\{  x\in X:\exists\delta>0\text{ }\forall y,z\in
B_{\delta}(x),\text{ }\underline{D}\left\{  i\in\mathbb{G}:\left\vert
f(T^{i}y)-f(T^{i}z)\right\vert \leq\varepsilon\right\}  \geq1-\varepsilon
\right\}  .
\]

\end{definition}

Note that $E^{f}=\cap_{\varepsilon>0}E_{\varepsilon}^{f}.$

\begin{lemma}
\label{invariant}Let $(X,T)$ be a $\mathbb{G-}$TDS. The sets $E^{f}$,
$E_{\varepsilon}^{f}$ are inversely invariant (i.e. $T^{-j}(E^{f})\subseteq
E^{f}$, $T^{-j}(E_{\varepsilon}^{f})\subseteq E_{\varepsilon}^{f}$ for all
$j\in\mathbb{G)}$ and $E_{\varepsilon}^{f}$ is open.
\end{lemma}

\begin{proof}
Let $j\in\mathbb{G},$ $\varepsilon>0,$ and $x\in T^{-j}E_{\varepsilon}^{f}.$
There exists $\eta>0$ such that if $d(T^{j}x,z)\leq\eta$ and $d(T^{j}%
x,y)\leq\eta$ then
\[
\underline{D}\left\{  i:\left\vert f(T^{i}y)-f(T^{i}z)\right\vert
\leq\varepsilon\right\}  \geq1-\varepsilon.
\]
There also exists $\delta>0$ such that if $d(x,y)\le\delta$ then
$d(T^{j}x, T^{j}y) \le\eta$. So, if $y,z \in B_{\delta}(x)$, then
$\underline{D}\left\{  i:\left\vert f(T^{i+j}y)-f(T^{i+j}z)\right\vert
\leq\varepsilon\right\}  \geq1-\varepsilon$$.$ We conclude that $x\in
E_{\varepsilon}^{f},$ and so $E_{\varepsilon}^{f}$ is inversely invariant. It
follows that $E^{f}$ is also inversely invariant.

Let $x\in E_{\varepsilon}^{f}$ and $\delta>0$ satisfy the defining property of
$E_{\varepsilon}^{f}.$ If $d(x,w)<\delta/2$ then $w\in E_{\varepsilon}^{f};$
indeed if $y,z\in B_{\delta/2}(w)$ then $y,z\in B_{\delta}(x).$ So,
$E_{\varepsilon}^{f}$ is open.
\end{proof}

It is not hard to see that $f-$mean sensitive systems have no $f-$mean
equicontinuity points. The proof of the following dichotomies is very similar
to the proof of Theorem \ref{sensi copy(1)}; we include it for completeness.

\begin{theorem}
\label{sensi} Let $f\in C(X).$ A transitive system is either $f-$almost mean
equicontinuous or $f-$mean sensitive. A minimal system is either $f-$mean
equicontinuous or $f-$mean sensitive.
\end{theorem}

\begin{proof}
First, we show that if $(X,T)$ is a transitive system then for every
$\varepsilon$, $E_{\varepsilon}^{f}$ is either empty or dense. Assume
$E_{\varepsilon}^{f}$ is non-empty and not dense. Then $U=X\diagdown
(cl(E_{\varepsilon}^{f}))$ is a non-empty open set. Since the system is
transitive and $E_{\varepsilon}^{f}$ is nonempty and open (by Lemma
\ref{invariant}) there exists $t\in\mathbb{G}$ such that $U\cap T^{-t}%
(E_{\varepsilon}^{f})$ is non empty. By Lemma \ref{invariant} we have that
$U\cap T^{-t}(E_{\varepsilon}^{f})\subset U\cap E_{\varepsilon}^{f}%
=\emptyset,$ a contradiction.

If $E_{\varepsilon}^{f}$ is non-empty for every $\varepsilon>0$ then we have
that $E^{f}=\cap_{n\geq1}E_{1/n}^{f}$ is a residual set; hence the system is
$f$-almost mean equicontinuous.

If there exists $\varepsilon>0$ such that $E_{\varepsilon}^{f}$ is empty, then
for any open set $U$ there exist $y,z\in U$ such that $\underline{D}\left\{
i\in\mathbb{G}:\left\vert f(T^{i}y)-f(T^{i}z)\right\vert \leq\varepsilon
\right\}  <1-\varepsilon;$ this means that $\overline{D}\left\{
i\in\mathbb{G}:\left\vert f(T^{i}y)-f(T^{i}z)\right\vert >\varepsilon\right\}
\geq\varepsilon.$ It follows that $(X,T)$ is $f-$mean sensitive.

Now suppose $(X,T)$ is minimal and $f$-almost mean equicontinuous. For every
$x\in X$ and every $\varepsilon>0$ there exists $t\in\mathbb{G}$ such that
$T^{t}x\in E_{\varepsilon}^{f}.$ Since $E_{\varepsilon}^{f}$ is inversely
invariant, we have $x\in E_{\varepsilon}^{f}$. So, $x\in E^{f}.$
\end{proof}

In the following result we use a technique from the topological Halmos
Von-Neumann Theorem (e.g. see Chapter 5.5 \cite{walters2000introduction}).

\begin{theorem}
\label{efmean}Let $(X,T)$ be a $\mathbb{G-}$TDS. Then $(X,T)$ is mean
equicontinuous if and only if it is $f-$mean equicontinuous for every $f\in
C(X)$ (i.e. $C(X)=C_{me}$)$.$
\end{theorem}

\begin{proof}
$\Rightarrow$

Let $(X,T)$ be mean equicontinuous, $f\in C(X)$ and $\varepsilon>0.$ Since $X$
is compact, $f$ is uniformly continuous; thus there exists $\delta^{\prime}%
\in(0,\varepsilon)$ such that if $d(x,y)\leq\delta^{\prime}$ then $\left\vert
f(x)-f(y)\right\vert \leq\varepsilon.$ Using that $(X,T)$ is mean
equicontinuous there exists $\delta>0$ such that if $d(x,y)\leq\delta$ then
\[
\underline{D}(i\in\mathbb{G}:d(T^{i}x,T^{i}y)\leq\delta^{\prime})\geq
1-\delta^{\prime}.
\]
This implies that%
\[
\underline{D}(i\in\mathbb{G}:\left\vert f(x)-f(y)\right\vert \leq
\varepsilon)\geq1-\varepsilon.
\]
Hence $(X,T)$ is $f-$mean equicontinuous.

$\Leftarrow$

Let $\left\{  f_{n}\right\}  $ be a collection of functions such that
$\left\vert f_{n}\right\vert \leq1$ and the closure of its linear span is
$C(X).$

Such a collection separates points of $X$, and so
\[
{\overline{d}}(x,y):=\sum_{n=1}^{\infty}\frac{\left\vert f_{n}(x)-f_{n}%
(y)\right\vert }{2^{n}}%
\]
is a metric on $X.$ Let $\varepsilon>0$ and choose $N$ so that $\sum
_{n=N+1}^{\infty}2/2^{n}\leq\varepsilon/2.$ There exists $\delta>0$ such that
if $d(x,y)\leq\delta$ then
\[
\left\vert f_{n}(x)-f_{n}(y)\right\vert \leq\varepsilon/2
\]
for every $1\leq n\leq N.$ Thus if $d(x,y)\leq\delta$ then
\[
\sum_{n=1}^{\infty}\frac{\left\vert f_{n}(x)-f_{n}(y)\right\vert }{2^{n}}%
\leq\varepsilon/2+\varepsilon/2=\varepsilon.
\]
So, the identity map from $(X,d)$ to $(X,{\overline{d}})$ is a bijective
continuous map on a compact metric space and hence is a homeomorphism. We
conclude that $d$ and ${\overline{d}}$ are equivalent metrics.

It is not hard to see that mean equicontinuity is invariant under a change of
equivalent metrics. We will show $(X,T)$ is mean equicontinuous with respect
to ${\overline{d}}$.

Let $\varepsilon>0$ and choose $N$ so that $\sum_{n=N+1}^{\infty}2/2^{n}%
\leq\varepsilon/2.$ There exists $\delta>0$ such that if $d(x,y)\leq\delta$
then
\[
\lim\sup_{n\rightarrow\infty}\frac{1}{\nu(F_n) }\int
_{F_{n}}\left\vert f_{m}(T^{i}x)-f_{m}(T^{i}y)\right\vert d\nu(i)\leq
\varepsilon/2
\]
for every $1\leq m\leq N.$ Since $d$ and ${\overline{d}}$ are equivalent
metrics, there exists $\eta>0$ such that if ${\overline{d}}(x,y)\leq\eta$,
then $d(x,y)\leq\delta$, and so
\begin{align*}
&  \lim\sup_{n\rightarrow\infty}\frac{1}{\nu(F_n) }%
\int_{F_{n}}{\overline{d}}(T^{i}x,T^{i}y)d\nu(i)\\
&  \leq\varepsilon/2+\lim\sup_{n\rightarrow\infty}\frac{1}{\nu(F_n)}\int_{F_{n}}\sum_{m=1}^{N}\frac{\left\vert f_{m}%
(T^{i}x)-f_{m}(T^{i}y)\right\vert }{2^{m}}d\nu(i)\\
&  \leq\varepsilon.
\end{align*}
So, $(X,T)$ is mean equicontinuous with respect to ${\overline{d}}$.
\end{proof}

\subsection{Weakly almost periodic functions}

In \cite{ellis1959equicontinuity}, Ellis defined two topological notions of
almost periodic functions in $C(X)$.

We consider $C(X)$ as a Banach space with the $\ell^{\infty}$ norm. Two
classic topologies studied on Banach spaces are: the strong topology (given by
the norm) and the weak topology (the coarsest topology such that each element
in the dual space is continuous). Since $X$ is compact $f_{n}\rightarrow f$ in
the weak topology if and only if $f_{n}(x)\rightarrow f(x)$ for every $x\in
X.$

\begin{definition}
Let $(X,T)$ be a $\mathbb{G-}$TDS. We say $f\in C(X)$ is \textbf{almost
periodic} if $\left\{  U^{j}f:j\in\mathbb{G}\right\}  \subset C(X)$ is
precompact with respect to the strong topology. We say $f\in C(X)$ is
\textbf{weakly almost periodic }(WAP) if $\left\{  U^{j}f:j\in\mathbb{G}%
\right\}  \subset C(X)$ is precompact with respect to the weak topology. If
every $f\in C(X)$ is WAP then we say $(X,T)$ is\textbf{ weakly almost
periodic}.
\end{definition}

A system is equicontinuous if and only if every function is almost periodic. A
minimal system is equicontinuous if and only if it is weakly almost periodic
\cite{ellis1959equicontinuity}; nonetheless there are examples of transitive
weakly almost periodic systems that are not equicontinuous. We will show these
systems are always mean equicontinuous.

\begin{lemma}
Let $(X,T)$ be a $\mathbb{G-}$TDS. If $f$ is almost periodic then it is mean equicontinuous.
\end{lemma}

\begin{proof}
Let $\varepsilon>0.$ There exists a finite set
$F\subset\mathbb{G}$ such that for every $j\in\mathbb{G}$ there exists
$i_{j}\in F$ such that $\left\vert U^{j}f-U^{i_{j}}f\right\vert _{\infty}%
\leq\varepsilon$ (where $U$ is the Koopman operator on $C(X)$)$.$ By uniform
continuity there exists $\delta>0$ such that if $d(x,y)\leq\delta$ then
$\left\vert U^{i}f(x)-U^{i}f(y)\right\vert \leq\varepsilon$ for all $i\in F.$
Thus, if $d(x,y)\leq\delta$ then for every $j$
\begin{align*}
&  \left\vert f(T^{j}x)-f(T^{j}y)\right\vert \\
&  \leq\left\vert f(T^{j}x)-f(T^{i_{j}}x)\right\vert +\left\vert f(T^{i_{j}%
}x)-f(T^{i_{j}}y)\right\vert +\left\vert f(T^{i_{j}}y)-f(T^{j}y)\right\vert \\
&  \leq3\varepsilon
\end{align*}
and thus
\begin{align*}
&  d_{f}^{\prime}(x,y)=\\
&  \lim\sup\frac{1}{\nu(F_n)}\int_{F_{n}}\left\vert
f(T^{j}x)-f(T^{j}y)\right\vert d\nu(j)\\
&  \leq3\varepsilon.
\end{align*}

By the equivalence of the metrics, $d_{f}(x,y)$ and $d_{f}^{\prime}(x,y)$, we
conclude that $f$ is mean equicontinuous.
\end{proof}


\begin{proposition}
Let $\mathbb{G}$ be a countable group and $(X,T)$ be a uniquely ergodic
$\mathbb{G}-$TDS. If $f$ is weakly almost periodic then it is mean equicontinuous.
\end{proposition}

\begin{proof}
Let $f\in C(X)$ be weakly almost periodic. As a consequence of the de
Leeuw-Glicksberg decomposition (see Theorem 1.51 in \cite{glasner2003ergodic})
we have that
\[
f=g+h,
\]
where $g$ is almost periodic and $\lim\frac{1}{\left\vert F_{n}\right\vert
}\sum_{i\in F_{n}}h(T^{i}x)=0$ for every $x.$ This implies $h$ is mean
equicontinuous. By the previous lemma $g$ is mean equicontinuous; we conclude
$f$ is mean equicontinuous.
\end{proof}

Using that transitive WAP systems are uniquely ergodic (Lemma 1.50 in
\cite{glasner2003ergodic}), the previous Proposition and Theorem \ref{efmean}
we obtain the following result.

\begin{corollary}
Let $\mathbb{G}$ be a countable group. If $(X,T)$ is a transitive WAP
$\mathbb{G}-$TDS then it is mean equicontinuous.
\end{corollary}

\section{Measure theoretic and topological results}

\label{hybrid}

\bigskip In this section we present hybrid results, that is, results that
reflect the topology as well as the measure theoretic structure of the sytem.

We remind the reader that given a metric space $X$ and a Borel probability
measure $\mu$ we denote the Borel sets with $\mathcal{B}_{X}$ and the Borel
sets with positive measure with $\mathcal{B}_{X}^{+}.$

We say $(X,\mu,T)$ is an \textbf{ergodic} $\mathbb{G-}$\textbf{TDS }if $(X,T)$
is a $\mathbb{G-}$TDS, $\mu$ a Borel probability measure and $(X,\mu,T)$ an
ergodic $\mathbb{G-}$MPS.

\subsection{$\mu-$Mean equicontinuity}

\begin{definition}
\label{mumeanequi}Let $(X,T)$ be a $\mathbb{G-}$TDS, $\mu$ an invariant Borel
probability measure$.$ We say $(X,T)$ is $\mu-$\textbf{mean equicontinuous} if
for every $\tau$ $>0$ there exists a compact set $M\subset X,$ with
$\mu(M)\geq1-\tau,$ such that for every $\varepsilon$ $>0$ there exists
$\delta>0$ such that whenever $x,y\in M$ and $d(x,y)\leq\delta$ then
\[
d_{b}(x,y)\leq\varepsilon.
\]

\end{definition}

\begin{remark}
\label{este} It is not hard to see that $d_{b}(\cdot,\cdot)$ is a Borel function.

Also, denoting the closed $\varepsilon-$balls of the pseudometric $d_{b}$ by
$B_{\varepsilon}^{b}(x),$ for every $\varepsilon>0$ and every $x\in X$ we have
that $B_{\varepsilon}^{b}(x)$ is $\mu-$measurable.
\end{remark}

Measure theoretic forms of sensitivity for $\mathbb{G-}$TDS have been studied
in \cite{Gilman1},\cite{Cadre2005375}, and \cite{mtequicontinuity}. In
particular in \cite{mtequicontinuity} it was shown that ergodic $\mathbb{G-}%
$TDS are either $\mu-$equicontinuous or $\mu-$sensitive. In \cite{weakeq}
$\mu-$mean sensitivity was introduced.

\begin{definition}
A $\mathbb{G-}$TDS $(X,T)$ is $\mu-$\textbf{mean sensitive} if there exists
$\varepsilon>0$ such that for every $A\in\mathcal{B}_{X}^{+}$ there exists
$x,y\in A$ such that%
\[
d_{b}(x,y)>\varepsilon.
\]

A $\mathbb{G-}$TDS $(X,T)$ is $\mu-$\textbf{mean expansive} if there exists
$\varepsilon>0$ such that $\mu\times\mu\left\{  (x,y):d_{b}(x,y)>\varepsilon
\right\}  =1.$
\end{definition}

\begin{theorem}
\label{strongsensitiv}\cite{weakeq}Let $(X,\mu,T)$ be an ergodic $\mathbb{G-}%
$TDS. The following are equivalent:

$1)$ $(X,T)$ is $\mu-$mean sensitive.

$2)$ $(X,T)$ is $\mu-$mean expansive$.$

$3)$ There exists $\varepsilon>0$ such that for almost every $x,$
$\mu(B_{\varepsilon}^{b}(x))=0.$

$4)(X,T)$ is not $\mu-$mean equicontinuous.
\end{theorem}

In the next subsection we will obtain a similar result for $\mu-f-$mean sensitivity.

\subsection{$\mu-f-$Mean equicontinuity}

For the definition of $d_{f}$ and $B_{\varepsilon}^{f}$ see Definition
\ref{meansensitive} and Definition \ref{meansensitive2}.

\begin{definition}
Let $(X,T)$ be a $\mathbb{G-}$TDS, $\mu$ a Borel probability measure and $f\in
L^{2}(X,\mu).$ We say $(X,T)$ is $\mu-f-$\textbf{mean equicontinuous} if for
every $\tau$ $>0$ there exists a compact set $M\subset X,$ with $\mu
(M)\geq1-\tau,$ such that for every $\varepsilon$ $>0$ there exists $\delta>0$
such that whenever $x,y\in M$ and $d(x,y)\leq\delta$ then
\[
\rho_{f}(x,y)\leq\varepsilon.
\]
In this case we say $f$ is $\mu-$\textbf{mean equicontinuous}. We denote the
set of $\mu-$mean equicontinuous functions by $H_{me}$.
\end{definition}

The following fact is well known (for a proof see \cite{weakeq}).

\begin{lemma}
\label{separable}Let $\mathcal{(}Y,d_{Y})$ be a metric space$.$ Suppose that
there is no uncountable set $A\subset Y$ and $\varepsilon>0$ such that
$d_{Y}(x,y)>\varepsilon$ for every $x,y\in A$ with $x\neq y,$ then $(Y,d_{Y})$
is separable.
\end{lemma}

\begin{theorem}
\label{dichof}Let $(X,\mu,T)$ be an ergodic $\mathbb{G-}$TDS and $f\in
L^{2}(X,\mu)$. The following are equivalent:

$1)$ $(X,T)$ is $\mu-f-$mean sensitive.

$2)$ $(X,T)$ is $\mu-f-$mean expansive$.$

$3)$ There exists $\varepsilon>0$ such that for almost every $x,$
$\mu(B_{\varepsilon}^{f}(x))=0.$

$4)(X,T)$ is not $\mu-f-$mean equicontinuous.
\end{theorem}

\begin{proof}
$1)\Leftrightarrow2)\Leftrightarrow3)$

Given by Theorem \ref{strongsensitive}.

The following proofs are similar to the proof of Theorem \ref{strongsensitiv}.

$2)\Rightarrow4)$

If $(X,T)$ is $\mu$-$f$-mean-expansive, then there is a set $Z\subset X\times
X$ s.t $\mu(Z)=1$ and for all $(x,y)\in Z$, we have $d_{f}(x,y)>\varepsilon$.
Suppose that $(X,T)$ is also $\mu$-$f$-mean equicontinuous. Then there is a
compact set $M\subseteq X$ with positive measure and $\delta>0$ such that if
$x,y\in M$ and $d(x,y)\leq\delta$, then $d_{f}(x,y)\leq\varepsilon$. Since $M$
is compact, it can be covered by finitely many $\delta/2$-balls. For at least
one of these balls $B$, we have $\mu(B\cap M)>0$ and so $((B\cap
M)\times(B\cap M))\cap Z\neq\emptyset.$ For any $(x,y)$ in this intersection,
we have $d_{f}(x,y)\leq\varepsilon$ and $d_{f}(x,y)>\varepsilon$, a contradiction.


$4)\Rightarrow3)$

Suppose $3)$ is not satisfied. By Lemma \ref{ecu} we have that for every
$n\in\mathbb{N}$ there exists a set of full measure $Y_{n}$
such that $\mu(B_{1/n}^{f}(x))>0$ for all $x\in Y_{n}.$ Let $Y:=\cap
_{n\in\mathbb{N}}Y_{n}$. If $(Y\diagup d_{f},d_{f})$ is not separable then by
Lemma \ref{separable} there exists $\varepsilon>0$ and an uncountable set $A$
such that for every $x,y\in A$ with $x\neq y$ we have that $B_{\varepsilon
}^{f}(x)\cap B_{\varepsilon}^{f}(y)=\emptyset.$ For any $1/n<\varepsilon$,
this is a contradiction to the fact that $\mu(B_{\varepsilon}^{f}(x))>0$ for
all $x\in Y_{n}$. Hence, $(Y\diagup d_{f},d_{f})$ is separable. Using Lusin's
Theorem we conclude $(X,T)$ is $\mu-f-$mean equicontinuous (this is done
exactly as in Proposition 19 in \cite{weakeq}).
\end{proof}


\begin{corollary}
\label{corhap}\bigskip Let $(X,\mu,T)$ be an ergodic $\mathbb{G-}$TDS. Then
$H_{ap}=H_{ms}^{c}=H_{me}.$
\end{corollary}

\begin{definition}
We say a function $f\in L^{2}(X,\mu)$ is simple if it is a linear combination
of indicator functions (of measurable sets) $1_{B}.$
\end{definition}

\bigskip In the following result, the equivalence of $1)$, $4)$ and $5)$ was
proved in \cite{weakeq} for $\mathbb{Z}^{d}$ systems. That proof used symbolic
systems and does not work for continuous groups like $\mathbb{R}^{d}$.

\begin{theorem}
\label{meds}\textbf{\ }Let $(X,\mu,T)$ be an ergodic $\mathbb{G-}$TDS. The
following conditions are equivalent:

$1)$ $(X,T)$ is $\mu-$mean equicontinuous

$2)$ $(X,T)$ is $\mu-1_{B}-$mean equicontinuous for every $B\in\mathcal{B}%
_{X}.$

$3)$ $(X,T)$ is $\mu-f-$mean equicontinuous for every $f\in L^{2}.$

$4)$ $(X,\mu,T)$ has discrete spectrum.

$5)$ $(X,T)$ is not $\mu-$mean sensitive.
\end{theorem}

\begin{proof}
$2)\Leftrightarrow3)$

By Corollary \ref{corhap}, $H_{me}=H_{ap}$ and thus by\ Proposition~\ref{prop}
it is a closed subspace of $L^{2}(X,\mu)$. Now use the fact that simple
functions are dense in $L^{2}$.

$3)\Leftrightarrow4)$

Follows from Theorem \ref{discrete} and Corollary \ref{corhap}.

$1)\Leftrightarrow5)$

This is part of Theorem \ref{strongsensitiv}.

$4)\Rightarrow1)$

If $(X,\mu,T)$ has discrete spectrum, then it is measure-theoretically
isomorphic to an isometry (by Theorem \ref{discrete}). Isometries are
equicontinuous and therefore $\mu-$mean equicontinuous. Now use the fact that
$\mu-$mean equicontinuity is an isomorphism invariant (see Proposition 28 in
\cite{weakeq}).

$1)\Rightarrow2)$

Let $B\in\mathcal{B}_{X}$ and $1 > \varepsilon>0.$

There exists a compact set $M_{1}$ with $\mu(M_{1})\geq1-\varepsilon/3$ and
$\varepsilon^{\prime}>0$ such that if $x,y\in M_{1}$ and $d(x,y)\leq
\varepsilon^{\prime}$ then $1_{B}(x)=1_{B}(y)$ (namely, $M_{1}$ is the union
of two compact sets, one approximating $B$ and the other approximating $B^{c}%
$, both from the inside). We may assume that $\varepsilon^{\prime}
\le(1/3)\varepsilon$.

Since $(X,T)$ is $\mu-$mean equicontinuous there exists a compact set $M_{2}$
with $\mu(M_{2})\geq1-\varepsilon/3$ and $\delta>0$ such that if $x,y\in
M_{2}$ and $d(x,y)\leq\delta$ then
\[
\underline{D}(i\in\mathbb{G}:d(T^{i}x,T^{i}y)\leq\varepsilon^{\prime}%
)\geq1-\varepsilon^{\prime}.
\]

Let $M=M_{1}\cap M_{2}\cap\left\{  \text{generic points for }M_{1}\right\}  $
and $x,y\in M.$ We have that $\mu(M)\geq1-\varepsilon$. Since $\varepsilon<1$,
we have
\[
1_{B}T^{i}x=1_{B}T^{i}y\mbox{ if and only if }|1_{B}T^{i}x-1_{B}T^{i}%
y|\leq\varepsilon.
\]
So, if $d(x,y)\leq\delta$ then
\begin{align*}
&  \underline{D}(i\in\mathbb{G}:|1_{B}T^{i}x-1_{B}T^{i}y|\leq\varepsilon)\\
&  =\underline{D}(i\in\mathbb{G}:1_{B}T^{i}x=1_{B}T^{i}y)\\
&  \geq\underline{D}(i\in\mathbb{G}:d(T^{i}x,T^{i}y)\leq\varepsilon^{\prime
}\text{ and }T^{i}x,T^{i}y\in M_{1})\\
&  \geq1-\varepsilon^{\prime}-2\varepsilon/3\\
&  \geq1-\varepsilon.
\end{align*}

We conclude that $(X,T)$ is $\mu-1_{B}-$mean equicontinuous.
\end{proof}

\section{Relationship to results of quasicrystals\label{quasi}}

Of particular interest in the mathematical theory of quasicrystals and
aperiodic order is studying long range order that Delone sets may exhibit;
Delone sets are uniformly discrete and relatively dense subsets of
$\mathbb{R}^{d}$ (see \cite{baake2013aperiodic} and
\cite{kellendonk2015mathematics} for recent general expositions that contain
the main definitions in this section). A Delone set is crystalline if it is
periodic (in all $d$ directions). These are the most ordered Delone sets. We
say a Delone set is quasicrystalline if it has pure point diffraction spectrum
(this notion is defined using Fourier transforms) and uniform patch frequency
(configurations of patches of points appear with a uniform frequency).
Periodic Delone sets always have pure point diffraction spectrum and uniform
patch frequency. Initially it was not known if non-periodic subsets could have
pure point diffraction spectrum; nowadays there are many examples, perhaps the
best known is the Delone set associated to the Penrose tiling
\cite{robinson1996dynamical}. To every Delone set we can associate an
$\mathbb{R}^{d}-$topological\ dynamical system ( uniquely ergodic when the set
has uniform patch frequency), where the \ $\mathbb{R}^{d}-$action is defined
by the shifts in the $d$ directions. In \cite{baake2007characterization}
Baake, Lenz and Moody showed that a Delone set is crystalline if and only if
the associated $\mathbb{R}^{d}$- topological dynamical system is
equicontinuous (this means that the family of the shifts is equicontinuous).
This motivates the following question; can quasicrystals be characterized with
a weaker forms of equicontinuity? Combining the work of several papers it is
known that a Delone set is quasicrystalline if and only if the associated
dynamical system is uniquely ergodic and has discrete spectrum
\cite{dworkin1993spectral, lee2002pure, baake2004dynamical}. Using this and
Theorem \ref{meds} we obtain that a Delone set is quasicrystalline if and only
if the associated dynamical system is uniquely ergodic and $\mu$-mean
equicontinuous (with $\mu$ the unique invariant measure).

\bigskip

There are other characterizations that use the Besicovitch pseudometric to
characterize discrete spectrum. In \cite{quasicrystals} it is proven that a
(minimal) uniquely ergodic Delone dynamical system has discrete spectrum if
and only if it is mean almost periodic. That is, for every $x\in X$ and
$\varepsilon>0$ there exists a syndetic set $S\subset\mathbb{R}^{d}$, such
that $d_{b}(T^{i}x,x)\leq\varepsilon$ for every $i\in S.$ This result is
different from ours in the following senses; besides the clear distinction
that mean almost periodicity and mean equicontinuity are different concepts we
see that this result applies only for minimal systems and the condition mean
almost periodicity does not depend on the measure. Our characterization using
$\mu-$mean equicontinuity requires the property of mean equicontinuity only on
sets with large measure. Nonetheless there is a relationship. It is not hard
to see that every minimal mean equicontinuous $\mathbb{G-}$TDS is mean almost
periodic$.$ Let $\varepsilon>0$ and $\delta>0$ the number such that if
$d(x,y)\leq\delta$ then $d_{b}(x,y)\leq\varepsilon.$ Since $(X,T)$ is minimal
every point is almost periodic, thus there exists a syndetic set
$S\subset\mathbb{R}^{d}$, such that $d(T^{i}x,x)\leq\delta$ for every $i\in
S;$ this implies $d_{b}(T^{i}x,x)\leq\varepsilon$ for every $i\in S.$ On
\cite{lenz2016autocorrelation} there is a similar characterization that uses
Bohr almost periodicity and applies to general $\mathbb{G-}$TDS.

\bibliographystyle{aabbrv}
\bibliography{acompat,camel}

\end{document}